\documentclass[11pt]{article}
\usepackage[margin=1.0in]{geometry}
\usepackage[utf8]{inputenc}
\usepackage[T1]{fontenc}
\usepackage[english]{babel}
\usepackage{lmodern}
\usepackage{graphicx}
\usepackage{wrapfig}
\usepackage{subfig}
\usepackage{caption}
\usepackage{paralist}
\usepackage{microtype}

\graphicspath{{./figures/}}

\newcommand{\titel}{Cut Tree Structures with Applications on Contraction-Based Sparsification}

\usepackage{algorithm}

\usepackage[noend]{algpseudocode}

\usepackage[usenames]{color}
\definecolor{hellblau}{rgb}{0.2,0.4,1}
\definecolor{dunkelblau}{rgb}{0,0,0.8}
\definecolor{dunkelgruen}{rgb}{0,0.5,0}
\usepackage[
pdftex,
colorlinks,
linkcolor=dunkelblau,
urlcolor=dunkelblau,
citecolor=dunkelgruen,
bookmarks=true,
linktocpage=true,
pdftitle={\titel},
pdfauthor={},
pdfsubject={}
]{hyperref} 
\usepackage{pdfpages}

\usepackage{amsmath}
\usepackage{amsthm}
\usepackage{amsfonts}
\usepackage{amssymb}
\theoremstyle{plain}
\newtheorem{satz}{Satz}[section]
\newtheorem{theorem}[satz]{Theorem}
\newtheorem{lemma}[satz]{Lemma}

\newtheorem{proposition}[satz]{Proposition}
\newtheorem{corollary}[satz]{Corollary}
\theoremstyle{remark}

\theoremstyle{definition}
\newtheorem{definition}[satz]{Definition}

\newtheorem{innercustomgeneric}{\customgenericname}
\providecommand{\customgenericname}{}
\newcommand{\newcustomtheorem}[2]{
  \newenvironment{#1}[1]
  {
   \renewcommand\customgenericname{#2}
   \renewcommand\theinnercustomgeneric{##1}
   \innercustomgeneric
  }
  {\endinnercustomgeneric}
}
\newcustomtheorem{customtheorem}{Theorem}
\newcustomtheorem{customlemma}{Lemma}
\newcustomtheorem{customproposition}{Proposition}
\newcustomtheorem{customcorollary}{Corollary}

\begin{document}
	\title{\titel}
	\author{On-Hei S. Lo\thanks{This research was supported by the DFG grant SCHM 3186/1-1.}\\ \textit{Institute of Mathematics}\\\textit{TU Ilmenau, Germany}
		\and Jens M. Schmidt\addtocounter{footnote}{-1}\footnotemark\\ \textit{Institute of Mathematics}\\\textit{TU Ilmenau, Germany}}
	
	\date{}
	\maketitle
	
	\begin{abstract}
		We introduce three new cut tree structures of graphs $G$ in which the vertex set of the tree is a partition of $V(G)$ and contractions of tree vertices satisfy sparsification requirements that preserve various types of cuts. Recently, Kawarabayashi and Thorup \cite{Kawarabayashi2015a} presented the first deterministic near-linear edge-connectivity recognition algorithm. A crucial step in this algorithm uses the existence of vertex subsets of a simple graph $G$ whose contractions leave a graph with $\tilde{O}(n/\delta)$ vertices and $\tilde{O}(n)$ edges ($n := |V(G)|$) such that all non-trivial min-cuts of $G$ are preserved. We improve this result by eliminating the poly-logarithmic factors, that is, we show a contraction-based sparsification that leaves $O(n/\delta)$ vertices and $O(n)$ edges and preserves all non-trivial min-cuts. We complement this result by giving a sparsification that leaves $O(n/\delta)$ vertices and $O(n)$ edges such that all (possibly not minimum) cuts of size less than $\delta$ are preserved, by using contractions in a second tree structure. As consequence, we have that every simple graph has $O(n/\delta)$ $\delta$-edge-connected components, and, if it is connected, it has $O((n/\delta)^2)$ non-trivial min-cuts. All these results are proven to be asymptotically optimal.
		
		By using a third tree structure, we give a new lower bound on the number of \emph{pendant pairs} (that is, pairs of vertices $v$, $w$ with $\lambda(v, w) = \min\{d(v),d(w)\}$). The previous best bound was given 1974 by Mader, who showed that every simple graph contains $\Omega(\delta^2)$ pendant pairs. We improve this result by showing that every simple graph $G$ with $\delta \geq 5$ or $\lambda \geq 4$ or $\kappa \geq 3$ contains $\Omega(\delta n)$ pendant pairs. We prove that this bound is asymptotically tight from several perspectives, and that $\Omega(\delta n)$ pendant pairs can be computed efficiently.
	\end{abstract}

	\section{Introduction}
	Edge-connectivity and the structure of (near-)minimum cuts of graphs have been studied intensively for the last 60 years. Many of the discovered structures like Gomory-Hu trees~\cite{Gomory1961}, cactus representations and the lattice of minimum $s$-$t$-cuts led to increasingly faster algorithms for recognizing, listing or counting various (near-)minimum cuts of graphs. In this paper, we propose three new cut tree structures that are more fine-grained than the ones known and that lead to new insights into edge-connectivity.

	\subsection{Previous Work}
	Recently, Kawarabayashi and Thorup~\cite{Kawarabayashi2015a} presented the first deterministic near-linear time min-cut algorithm for simple graphs. They showed that certain vertex sets of the input graph can be contracted in near-linear time such that $\tilde{O}(n/\delta)$ vertices and $\tilde{O}(n)$ edges are left and all non-trivial min-cuts are preserved. This contraction-based sparsification implies a min-cut algorithm in near-linear time, as the contractions leave a graph on which Gabow's algorithm~\cite{Gabow1995} can be applied, which runs itself in time $\tilde{O}(\lambda m)$. Subsequently, Henzinger, Rao and Wang~\cite{Henzinger2017} obtained an improved variant of~\cite{Kawarabayashi2015a} with running time $O(m \log^2 n \log \log^2 n)$ by replacing its diffusion subroutine with a flow-based one, but again relying on contraction-based sparsification.
	
	Here, we will focus on the existential question to which extent such contraction-based sparsifiers can be improved and lay out the structural foundation of this concept. The importance of this question does not only stem from the fact that the currently fastest methods are using these sparsifiers, it may also serve for constructing faster contraction-based algorithms.
	
	From a different angle, min-cuts can be recognized by the well-known, simple and widely used minimum cut algorithm of Nagamochi and Ibaraki~\cite{Nagamochi1992b}, which refines the work of Mader~\cite{Mader1973,Mader1971b} in the early 70s, and was simplified by Stoer and Wagner~\cite{Stoer1997} and Frank~\cite{Frank1994}. The key approach of this algorithm is to iteratively contract a pendant pair of the input graph in near-linear time by using \emph{maximal adjacency orderings} (also known as \emph{maximum cardinality search}~\cite{Tarjan1984}). Having done that $n-1$ times, one can obtain a minimum cut by considering the minimum degree of all intermediate graphs.
	
	This algorithm motivates the study of pendant pairs: How many (distinct) pendant pairs does a graph with a given minimum degree have? If there are many and, additionally, these could be computed efficiently, this would lead immediately to an improvement of the running time of the Nagamochi-Ibaraki-algorithm (we conjecture that $\omega(1)$ pendant pairs could be computed in linear time, which would imply such a speed-up). Thus, we aim for the fundamental and natural question of finding a good lower bound on the number of distinct pendant pairs in graphs with a given minimum degree.
	
	As early as 1973, and originally motivated by the structure of minimally $k$-edge-connected graphs, Mader proved that every graph with minimum degree $\delta \geq 1$ contains at least one pendant pair~\cite{Mader1973}.
	Later, Mader improved his result by showing that every simple graph with minimum degree $\delta$ contains $\Omega(\delta^2)$ pendant pairs~\cite{Mader1974a}, which is the best result known so far.
	
	We will solve both questions, about the existence of lower bounds and of good contraction-based sparsifiers, by proposing several new cut tree structures of graphs. We will mainly consider simple graphs, as these allow us to prove the strongest bounds (we give an example that shows that all bounds for multigraphs must be considerably weaker).

	\subsection{Our Results}
	In this paper, we improve the bounds of the contraction-based sparsifier of Kawarabayashi and Thorup by eliminating its poly-logarithmic factors. Hence, every simple graph can be sparsified by contractions of vertex subsets such that $O(n/\delta)$ vertices and $O(n)$ edges are left and every non-trivial min-cuts is preserved. We also show that there are vertex subsets whose contraction leaves only $O(n)$ edges and $O(n/\delta)$ vertices such that every cut of size smaller than $\delta$ is preserved. While the first result has the benefit of handling non-trivial min-cuts of size $\delta$, the second has the benefit of dealing with all cuts of size less than $\delta$ and is not restricted to non-trivial cuts; in this sense, these results are incomparable.
	
	Kawarabayashi and Thorup showed also that every connected simple graph has $O((n^2 \log^c n)/\delta^2)$ non-trivial minimum cuts for some constant $c$~\cite[Corollary~4]{Kawarabayashi2015a}. We improve this by eliminating the poly-logarithmic factor to the asymptotically optimal bound $O((n/\delta)^2)$. We also show that every simple graph $G$ with $\delta > 0$ has $O(n/\delta)$ many $\delta$-edge-connected components, which appears to be unknown so far. All results are proven to be asymptotically optimal.
	
	We further improve Mader's lower bound $\Omega(\delta^2)$ on the number of pendant pairs by showing that every simple graph that satisfies $\delta \geq 5$ or $\lambda \geq 4$ or $\kappa \geq 3$ contains $\Omega(\delta n)$ pendant pairs; this exhibits a dependency on $n$ instead of $\delta$, which is usually much larger. We prove that this result is tight with respect to the order of the bound and with respect to every assumption. All cut tree structures that we use can be computed efficiently. In particular, we show how to compute $\Omega(\delta n)$ pendant pairs in time $O(n \theta_{flow})$, where $\theta_{flow}$ is the time needed for finding a maximum $s$-$t$-flow.

	\subsection{Technical Overview}
	The proofs are inspired by the work of Mader~\cite{Mader1974a} and Cai~\cite{Cai1993}. We propose a general framework that unifies the three cut tree structures used in this paper. The framework may also be of independent interest. Given a non-crossing family of cuts in a graph $G$ such that the cuts \emph{cover} some binary relation on the vertex set $V(G)$ (meaning that every vertex pair of the relation is separated by at least one cut in the family), we consider a tree $T$ that satisfies the following properties: (i) $V(T)$ is a partition of $V(G)$; (ii) the binary relation is covered by the cuts of $G$ that correspond to the edges of $T$; (iii) $T$ is minimal in the sense that (ii) does not hold anymore when any edge of $T$ is contracted.
	
	We will use the binary relation that consists of the vertex pairs that are separated by some non-trivial min-cut, that are non-$\delta$-edge-connected and that are non-pendant, respectively. For these choices, we can show that the average size of the vertices of $T$ is $\Omega(\delta)$. This will allow us to count pendant pairs and obtain the desired lower bound on their number. Moreover, this implies that the number of vertices of $T$ is $O(|V(G)|/\delta)$. If every edge-cut that corresponds to an edge in $T$ contains $O(\delta)$ edges, we can contract the vertices of $T$ to obtain a graph having $O(|V(G)|/\delta)$ vertices and $O(|V(G)|)$ edges and hence achieve our sparsification results.

	\section{Preliminaries}
	All graphs considered in this paper are non-empty, finite, unweighted and undirected. Let $G:=(V,E)$ be a graph. \emph{Contracting} a vertex subset $X \subseteq V$ identifies all vertices in $X$ and deletes occurring self-loops (we do not require that $X$ induces a connected graph in $G$).
	
	For non-empty and disjoint vertex subsets $X,Y \subset V$, let $E_G(X,Y)$ denote the set of all edges in $G$ that have one endvertex in $X$ and one endvertex in $Y$. Let further $\overline{X} := V-X$, $d_G(X,Y):=|E_G(X,Y)|$ and $d_G(X):=|E_G(X,\overline{X})|$; if $X=\{v\}$ for some vertex $v\in V$, we simply write $E_G(v,Y)$, $d_G(v,Y)$ and $d_G(v)$. A subset $\emptyset \neq X \subset V$ of a graph $G$ is called a \emph{cut} of $G$. Let a cut $X$ of $G$ be \emph{trivial} if $|X|=1$ or $|\overline{X}|=1$. Let the \emph{length} and \emph{size} of a path be the number of its edges and vertices, respectively. Let $\delta(G) := \min_{v\in V} d_G(v)$ be the \emph{minimum degree} of $G$. For a vertex $v \in G$, let $N_G(v)$ be the set of neighbors of $v$ in $G$.
	
	For two vertices $v,w\in V$, a \emph{$v$-$w$-cut} is a vertex set $X\subseteq V$ such that exactly one of $\{v,w\}$ is in $X$. Let $\lambda_G(v,w)$ be the minimum $d_G(X)$ over all $v$-$w$-cuts $X$. Two vertices $v,w\in V$ are called \emph{$k$-edge-connected} if $\lambda_G(v,w) \geq k$. For any $k$, the $k$-edge-connectivity relation on vertices is symmetric and transitive, and thus its reflexive closure is an equivalence relation that partitions $V$; let the \emph{$k$-edge-connected components} be the blocks of this partition. The \emph{edge-connectivity} $\lambda(G)$ of $G$ is $\max_{v,w\in V} \lambda_G(v,w)$. Let $\kappa(G)$ be the \emph{vertex-connectivity} of $G$.
	
	We call a pair $\{v,w\}$ of vertices \emph{pendant} if $\lambda_G(v,w) = \min\{d_G(v),d_G(w)\}$. In order to increase readability, we will omit subscripts whenever the graph is clear from the context. Two cuts $X$ and $Y$ \emph{cross} if $X \cap Y$, $X \cap \overline{Y}$, $\overline{X} \cap Y$ and $\overline{X} \cap \overline{Y}$ are non-empty. The following two lemmas adapt the well-known uncrossing technique, which allows to replace crossing minimum cuts by non-crossing ones such that the set of edges covered remains the same.
	
	\begin{lemma}[\cite{Gomory1961}]\label{lem:noncrossing}
		Let $X$ be a minimum $a$-$b$-cut of $G$, where $a, b\in V(G)$. Let $s,t \in X$ and let $Y$ be a minimum $s$-$t$-cut of $G$ that crosses $X$. Then $X \cap Y$ or $X - Y$ is a minimum $s$-$t$-cut.
	\end{lemma}
	
	\begin{lemma}[\cite{Dinits1976}]\label{lem:strongernc}
		Let $X, Y$ be two crossing min-cuts of $G$. Then $X \cap Y$ and $X - Y$ are min-cuts.
	\end{lemma}
	\begin{proof}
		Since $2\lambda = d(X) + d(Y) \geq d(X \cap Y) + d(X \cup Y) \geq 2\lambda$, $X \cap Y$ is a min-cut. Similarly, $X - Y$ is a min-cut.
	\end{proof}
	
	Let $T$ be a tree. In this paper, we call the vertices of $T$ \emph{blocks}, as we will require them to form a partition of the vertex set of a graph $G$. For any tree edge $AB \in E(T)$, let $C_{AB}$ be the union of the blocks that are contained in the component of $T-AB$ containing $A$, and let $c(AB) := d_G(C_{AB})$. We use $V_k$ to denote the set of blocks of $T$ having degree $k$ in $T$ and $V_{>k}$ to denote $\bigcup_{j>k}  V_{j}$. We call the blocks in $V_1$ \emph{leaf blocks}. In $T$, the set $V_2$ induces a family of disjoint paths; we call each such path a \emph{2-path}. For every block $A \in V_2$, let $A$ be in $V_2^{in}$ if all of its neighbors are also in $V_2$ and otherwise in $V_2^{out}$. The blocks in $V_2^{out}$ are exactly the blocks at the ends of 2-paths.
	
	\begin{lemma}\label{lem:leafbound}
		Let $T$ be a tree. If $|V(T)|>1$, then $|V_{>2}| \leq |V_1|-2$ and $|V_2^{out}|\leq 4|V_1|-6$.
	\end{lemma}
	\begin{proof}
		As $T$ is a tree, $2(|V(T)|-1) = \sum_{A \in V(T)} d_T(A) \geq |V_1| + 2|V_2| + 3|V_{>2}|$, which yields $|V_{>2}| \leq |V_1|-2$. Since every 2-path contains at most two blocks in $V_2^{out}$ and contracting every 2-path along with one of its neighbors gives a tree $T'$ with $V(T') =  V_1 \cup  V_{>2}$, we have $|V_2^{out}| \leq 2E(T') = 2(|V_1|+|V_{>2}|-1)$. Thus, $|V_2^{out}|\leq 4|V_1|-6$.
	\end{proof}

	\section{Contraction-Based Sparsification}\label{sec:improvement}
	Kawarabayashi and Thorup contract vertex subsets of $G$ such that $O((n \log^c n)/\delta)$ vertices and $O(n \log^c n)$ edges remain for some constant $c$ and all non-trivial min-cuts are preserved. Next, we show the existence of two such contraction-based sparsifications by introducing tree structures that either cover non-trivial min-cuts or cover cuts of size less than $\delta$. In other words, contracting the vertices of the former (latter) tree will preserve all non-trivial min-cuts (cuts of size less than $\delta$).

	\subsection{Preserving Non-Trivial Min-Cuts}\label{sec:nt-tree}
	We seek vertex sets whose contractions preserve all non-trivial min-cuts. We show that it is possible to find them greedily: We start with the partition $\{V\}$; whenever there is a block of the partition containing two vertices that are separated by some non-trivial min-cut, we split this block into two blocks. In particular, we can find non-crossing non-trivial min-cuts for the splits, which shows an underlying tree structure (we refer to Appendix~\ref{sec:appcbs} for thorough proofs). We state the definition of this tree as follows.
	
	\begin{definition}
		A \emph{non-trivial min-cut tree} $T$ of a graph $G=(V,E)$ is a tree whose vertex set partitions $V$ such that
		\begin{compactitem}
			\item[(i)] for every edge $AB \in T$, $C_{AB}$ is a non-trivial min-cut of $G$, and
			\item[(ii)] for every two vertices $a$ and $b$ that are separated by some non-trivial min-cut of $G$, there is an edge $AB \in T$ such that $C_{AB}$ separates $a$ and $b$.
		\end{compactitem}
	\end{definition}
	
	Condition~(i) implies that, for $\lambda \neq 0$, no leaf block is a singleton and that $c(AB)=\lambda$ for all $AB \in T$. Condition~(ii) implies that all non-trivial minimum cuts will be preserved if every block is contracted.
	
	Non-trivial min-cut trees do not exist for every graph, even if we restrict ourselves to graphs having no isolated vertex. For instance, consider a cycle of length at least four. As every edge is contained in a non-trivial min-cut, every leaf block $A$ of a non-trivial min-cut tree $T$ is an independent set of size at least two in $G$ due to Condition~(i). Then the tree edge $AB \in E(T)$ satisfies $c(AB) \geq 4$, which contradicts Condition~(i). However, we can show that non-trivial min-cut trees exist for all simple graphs $G$ with $\lambda(G) \neq 0, 2$. We refer to Appendix~\ref{sec:appcbs} for proofs.
	
	\begin{proposition}\label{prop:computationn}
		Let $G$ be a simple graph with $\lambda(G) \neq 0, 2$. Then a non-trivial min-cut tree $T$ of $G$ can be computed in time $O(n\theta_{flow})$.
	\end{proposition}
	
	To proof of the improved sparsification need the following lemmas, which essentially show that the blocks of a non-trivial min-cut tree have average size $\Omega(\delta)$. We refer to Appendix~\ref{sec:appcbs} for proofs.
	
	\begin{lemma}\label{lem:leafn}
		Let $T$ be a non-trivial min-cut tree of a simple graph $G$. Then every leaf block $A$ of $T$ satisfies $|A| \geq \delta(G)$.
	\end{lemma}
		
	\begin{lemma}\label{lem:bignonsingletonsn}
		Let $T$ be a non-trivial min-cut tree of a simple graph $G$. Let $A'A$, $AB$, $BB'$ be edges in $T$ such that $A,B \in V_2$. If $|A|+|B|>2$, $|A| + |B| \geq \delta(G)/2$.
	\end{lemma}
		
	\begin{lemma}\label{lem:rstarn}
		Let $T$ be a non-trivial min-cut tree of a simple graph $G$. Let $A$ be a block in $V_r$ with neighborhood $B_1,\dots,B_r\in V_2$ in $T$ such that $|A|=|B_1|= \dots = |B_r| = 1$. Then $\delta(G) \leq r^2 + r$.
	\end{lemma}
	
	Now the main theorem for non-trivial min-cut trees states the following.
	
	\begin{theorem}\label{thm:contractionn}
		Let $G$ be a simple graph with $\delta > 0$. If $T$ is a non-trivial min-cut tree of $G$, then $O(n/\delta)$ vertices and $O(n)$ edges are left if all blocks of $T$ are contracted.
	\end{theorem}
	\begin{proof}
		We can assume $\delta \geq 7$ and $|V(T)|>1$, as otherwise there are obviouisly at most $O(n/\delta)$ vertices left after the contractions of $V(T)$. In particular, we have $V_0 = \emptyset$. Since $\delta \geq 7$, Lemma~\ref{lem:rstarn} implies that there are no distinct blocks $B_1, B_2, B_3 \in V_2$ satisfying $B_1B_2, B_2B_3 \in E(T)$. We conclude by Lemma~\ref{lem:bignonsingletonsn} that, for every 2-path $P$ with $|V(P)| \geq 3$, $\sum_{S\in V(P)} |S| = |V_2^{in}\cap V(P)| \cdot \Omega(\delta)$. Now we consider the number of vertices of $G$,
		\begin{align*}
		n &=\sum_{S\in V(T)} |S|\\
		&= \sum_{S\in V_1\cup V_{>2}} |S| +\sum_{\textrm{2-path }P}\left(\sum_{S\in V(P)}|S|\right)\\
		&\geq |V_1|\cdot \Omega(\delta)+\sum_{\textrm{2-path }P, |V(P)| \geq 3}\left(|V_2^{in}\cap V(P)| \cdot \Omega(\delta)\right) \tag{by Lemma~\ref{lem:leafn}}\\
		&= (|V_1|+|V_2^{out}|+|V_{>2}|)\cdot \Omega(\delta) + |V_2^{in}| \cdot \Omega(\delta) \tag{by Lemma~\ref{lem:leafbound}}\\
		&= |V(T)| \cdot \Omega(\delta).
		\end{align*}
		Therefore, $|V(T)| = O(n/\delta)$ many vertices and at most $(|V(T)| - 1) \cdot \lambda = O(n \lambda/\delta) \leq O(n)$ many edges will be left if all blocks of $T$ are contracted.
	\end{proof}
	
	Since contractions do not decrease the edge-connectivity and all non-trivial minimum cuts are preserved when the blocks of a non-trivial min-cut tree are contracted, the number non-trivial min-cuts in $G$ is bounded above by the number of min-cuts in the graph obtained by these contractions. Hence, by Theorem~\ref{thm:contractionn} and a well-known result in \cite{Dinits1976}, which states that the number of minimum cuts in any connected graph $H$ is at most $O(|V(H)|^2)$, we have the following result.
	
	\begin{theorem}\label{thm:numbermincutsn}
		Every simple graph $G$ with $\lambda \neq 0, 2$ has $O((n/\delta)^2)$ many non-trivial minimum cuts.
	\end{theorem}

	\subsection[Preserving Cuts of Size less than delta]{Preserving Cuts of Size less than $\boldsymbol{\delta}$}\label{sec:ec-tree}
	To preserve all cuts of size less than $\delta$, we consider a second cut tree structure defined as follows.
	
	\begin{definition}
		A \emph{$k$-edge-connectivity tree} $T$ of a graph $G=(V,E)$ is a tree whose vertex set partitions $V$ such that
		\begin{compactitem}
			\item[(i)] every two distinct vertices in a common block of this partition are $k$-edge-connected, and
			\item[(ii)] for every edge $AB \in T$, there are vertices $a^\ast \in A$ and $b^\ast \in B$ such that $c(AB) = \lambda_G(a^\ast,b^\ast) < k$.
		\end{compactitem}
	\end{definition}
	
	Thus, the blocks of any $k$-edge-connectivity tree are exactly the $k$-edge-connected components. We will give only a brief account of this tree structure, as we will give a stronger tree structure in Section~\ref{sec:LBPP} whose results (see Theorem~\ref{thm:independent}) will imply the following theorem; we refer to Section~\ref{sec:PTrelax} for the interplay between these two tree structures.
	
	\begin{theorem}\label{thm:fewcomponents}
		Let $G$ be a simple graph with $\delta > 0$. Let $T$ be a $\delta(G)$-edge-connectivity tree. Then $T$ has at most $O(n/\delta)$ blocks. In other words, every simple graph $G$ with $\delta > 0$ has $O(n/\delta)$ many $\delta(G)$-edge-connected components.
	\end{theorem}
	
	By Theorem~\ref{thm:fewcomponents} and the fact that every edge-cut that corresponds to an edge of the $\delta$-edge-connectivity tree has $O(\delta)$ edges, we have the following theorem (compare to the proof of Theorem~\ref{thm:contractionn}).
	
	\begin{theorem}\label{thm:contraction}
		Contracting every $\delta$-edge-connected component of a simple graph $G$ satisfying $\delta > 0$ leaves $O(n/\delta)$ vertices and $O(n)$ edges.
	\end{theorem}
	
	By the similar arguements for showing Theorem~\ref{thm:numbermincutsn}, we have the following implied by Theorem~\ref{thm:contraction}.

	\begin{theorem}\label{thm:numbermincuts}
		Every simple graph $G$ that satisfies $0 < \lambda(G) < \delta(G)$ has $O((n/\delta)^2)$ minimum cuts.
	\end{theorem}

	Combining Theorems~\ref{thm:numbermincutsn} and~\ref{thm:numbermincuts}, we have the following:
	
	\begin{theorem}\label{thm:numbernmincuts}
		Every connected simple graph $G$ has $O((n/\delta)^2)$ non-trivial min-cuts.
	\end{theorem}
	\begin{proof}
		For $0 < \lambda = \delta \leq 2$, we apply the well-known result in \cite{Dinits1976}. For $\lambda = \delta > 2$, we apply Theorem~\ref{thm:numbermincutsn}. For $0 < \lambda < \delta$, we apply Theorem~\ref{thm:numbermincuts}.
	\end{proof}
	
	\subsection{Tightness}\label{sec:tightness1}
	We prove that the our results are tight. The following graph shows that the bounds of Theorems~\ref{thm:contractionn},~\ref{thm:fewcomponents} and~\ref{thm:contraction} (vertex- and edge-bounds) and of Theorems~\ref{thm:numbernmincuts} and~\ref{thm:numbermincuts}, are asymptotically tight. Let $n \geq 3(\delta+1)$, $\delta \geq 2$ and assume that $n$ is a multiple of $\delta+1$ (the last assumption can be avoided by a simple modification of the construction). Then the graph $G$ obtained from the cycle on $n/(\delta+1)$ vertices by replacing all vertices with a copy of $K_{\delta+1}$ shows tightness. Although this graph satisfies always $\lambda=2$, it can be readily generalized to graphs $G'$ having larger and even $\lambda$ without losing tightness (for $\lambda < \delta/2$). To do so, obtain $G'$ from $G$ by adding $\lambda/2-1$ cycles (each of size $n/(\delta+1)$) such that all these $\lambda/2$ cycles are vertex-disjoint and they visit the $n/(\delta+1)$ cliques $K_{\delta+1}$ in the same order.
	
	For Theorems~\ref{thm:numbernmincuts} and~\ref{thm:numbermincuts}, the assumption $\lambda \neq 0$ is (not only technically) necessary, as shown by the graph having $n$ isolated vertices, which has exponentially many non-trivial min-cuts. For Theorem~\ref{thm:numbermincuts}, the assumption $\lambda < \delta$ is necessary due to the complete graphs $K_n$.

	\section{Lower Bound for Pendant Pairs}\label{sec:LBPP}
	We introduce our third tree structure, called \emph{pendant tree}, in which all non-pendant pairs are covered by cuts. This tree will imply a new and tight lower bound on the number of pendant pairs. All omitted proofs of this section can be found in Appendix~\ref{sec:LBPPA}.
	
	\begin{definition}\label{def:pendant}
		A \emph{non-pendant-pair covering tree}, or simply \emph{pendant tree}, $T$ of a graph $G=(V,E)$ is a tree whose vertex set partitions $V$ such that
		\begin{compactitem}
			\item[(i)] every two distinct vertices in a common block of this partition are pendant,
			\item[(ii)] for every edge $AB \in T$, there are vertices $a \in A$ and $b \in B$ such that $\{a,b\}$ is non-pendant, and
			\item[(iii)] for every edge $AB \in T$, there are vertices $a^\ast \in A$ and $b^\ast \in B$ such that $c(AB) = \lambda_G(a^\ast,b^\ast)$.
		\end{compactitem}
	\end{definition}
	
	We first show for every edge $AB \in E(T)$ and every vertex $a_{max}$ of $A$ of maximum degree that $c(AB)$ cannot be too large. By definition, there are vertices $a \in A$ and $b \in B$ such that $\lambda(a,b) < \min\{d(a),d(b)\}$. Since $\{a,a_{max}\}$ and $\{b,b_{max}\}$ are pendant, a minimum $a$-$b$-cut can neither separate $a$ from $a_{max}$ nor $b$ from $b_{max}$. Hence, $\lambda(a_{max},b_{max}) \leq \lambda(a,b) < \min\{d(a),d(b)\} \leq \min\{d(a_{max}),d(b_{max})\}$. Now, let $a^\ast \in A$ and $b^\ast \in B$ be such that $c(AB) = \lambda(a^\ast,b^\ast)$ due to Condition~(iii). By transitivity of $\lambda$, we have $\lambda(a_{max},b_{max}) \geq \min\{\lambda(a_{max},a^\ast),\lambda(a^\ast,b^\ast),\lambda(b^\ast,b_{max})\} = \min\{d(a^\ast),\lambda(a^\ast,b^\ast),d(b^\ast)\} = c(AB)$, where the first equality follows from the fact that $\{a_{max},a^\ast\}$ and $\{b_{max},b^\ast\}$ are pendant. Therefore, $c(AB) < d(a_{max})$ and the following lemmas hold.
	
	\begin{lemma}\label{lem:representative}
		Let $AB$ be an edge of a pendant tree $T$. Then $\{a_{max},b_{max}\}$ is non-pendant.
	\end{lemma}
	
	\begin{lemma}\label{lem:maxdeg}
		Let $AB$ be an edge of a pendant tree $T$ and let $a_{max}$ be a vertex in $A$ of maximum degree. Then $c(AB) < d(a_{max})$.
	\end{lemma}
	
	In order to compute a pendant tree of $G$, we may partition $V$ by applying Lemma~\ref{lem:noncrossing} iteratively to every non-pendant pair; however, this would result in a running time of $O(n^2 \theta_{flow})$. Instead, we show that a runtime of $O(n \theta_{flow})$ suffices.
	
	\begin{proposition}\label{prop:computation}
		A pendant tree $T$ of a graph $G$ can be computed in time $O(n \theta_{flow})$.
	\end{proposition}
	
	In particular, every graph has a pendant tree.

	\subsection{Large Blocks of Degree 1 and 2}\label{sec:largeblockPT}
	Following the line of arguments used for the tree structures of Section~\ref{sec:improvement}, we will prove that the leaf blocks of pendant trees, as well as the blocks that are contained in 2-paths, are of average size $\Omega(\delta)$.
	
	\begin{lemma}\label{lem:leaf}
		Every leaf block $A$ of a pendant tree $T$ in a simple graph $G$ satisfies $|A| > \delta(G)$.
	\end{lemma}
	
	Let $a_{max}$ be a vertex of maximal degree in a leaf block $A$ with neighbor $B$ in $T$. Since $c(AB) < d(a_{max})$, $A$ must actually contain a vertex that has all its neighbors in $A$, as otherwise each neighbor $u$ of $a_{max}$ would contribute at least one edge to the edge-cut, by either the edge $a_{max} u$ or an incident edge of $u$. This gives the following corollary of Lemma~\ref{lem:leaf}, which was first shown by Mader.
	
	\begin{corollary}[\cite{Mader1974a}]\label{cor:Madervertex}
		Every leaf block $A$ of a pendant tree $T$ in a simple graph $G$ contains a vertex $v$ with $N(v) \subseteq A$. Hence, every pair in $\{v\} \cup N(v)$ is pendant.
	\end{corollary}
	
	This already implies that simple graphs contain $\binom{\delta+1}{2} = \Omega(\delta^2)$ pendant pairs. Note that Lemma~\ref{lem:leaf} and Corollary~\ref{cor:Madervertex} do not hold for graphs having parallel edges: for example, consider a block $A$ that consists of two vertices of degree $\delta$, which are joined by $\delta-1$ parallel edges. However, even if the graph is not simple, a leaf block $A$ must always contain at least two vertices due to Lemma~\ref{lem:maxdeg}.
	
	\begin{corollary}\label{cor:nosingletonleaf}
		Every leaf block of a pendant tree of a graph contains at least two vertices.
	\end{corollary}
	
	In simple graphs, we thus know that leaf blocks give us a large number of pendant pairs. Since $T$ is a tree, the number of leaf blocks is exactly determined by the number of blocks of size at least 3, namely $|V_1| = \sum_{A \in  V_{>2}} (d_T(A)-2) + 2$. Thus, in order to prove a better lower bound on the number of pendant pairs, we have to consider the case that there would be too many small blocks of size $o(\delta)$ in 2-paths. We can show that (i) for every two adjacent blocks $A$ and $B$ in a 2-path with $|A|+|B|>2$, we have $|A|+|B|\geq\delta-1=\Omega(\delta)$ and (ii) under certain assumptions, if $P$ is a subpath of a 2-path such that all blocks of $P$ are singletons, then $P$ contains at most two blocks. This implies that the bad situation of many small blocks of size $o(\delta)$ cannot occur and we can have that the blocks are of average size $\Omega(\delta)$. We have the following lemma.
	
	\begin{lemma}\label{lem:big2paths}
		Let $T$ be a pendant tree of a simple graph $G$ satisfying $\delta(G) \geq 5$ or $\lambda(G) \geq 4$ or $\kappa(G) \geq 3$. Let $P$ be a 2-path of $T$. Then $\sum_{S\in V(P)} |S| \geq (|V(P)|-2)\frac{\max\{4,\delta(G)\}}{3}+2$.
	\end{lemma}
	
	By Lemma~\ref{lem:leaf} and Lemma~\ref{lem:big2paths}, we know that the sizes of leaf blocks and inner blocks of 2-paths are large enough for our need. Now, by Lemma~\ref{lem:leafbound}, the number of these large blocks is a constant fraction of the total number of blocks, which implies that all blocks are of average size $\Omega(\delta)$.

	\subsection[Many Pendant Pairs]{$\boldsymbol{\Omega(\delta n)}$ Many Pendant Pairs}
	We use the previous results on large blocks to obtain our main Theorems~\ref{thm:independent} and~\ref{thm:main} for this section. While the latter shows the existence of $\Omega(\delta n)$ pendant pairs, the former gives an upper bound on the number of blocks of pendant trees. We defer to Appendix~\ref{sec:LBPPA} for all proofs.
	
	\begin{theorem}\label{thm:independent}
		Let $G$ be a simple graph that satisfies $\delta(G) \geq 5$ or $\lambda(G) \geq 4$ or $\kappa(G) \geq 3$. Let $T$ be a pendant tree of $G$. Then $T$ has at most $\frac{12}{\delta + 12} n = O(n/\delta)$ blocks.
	\end{theorem}
	
	\begin{theorem}\label{thm:main}
		Let $G$ be a simple graph that satisfies $\delta(G) \geq 5$ or $\lambda(G) \geq 4$ or $\kappa(G) \geq 3$. Then $G$ contains at least $\frac{1}{30}\delta n = \Omega(\delta n)$ pendant pairs.
	\end{theorem}

	\subsection[delta-Edge-Connectivity Tree and Pendant Tree]{$\boldsymbol{\delta}$-Edge-Connectivity Tree and Pendant Tree}\label{sec:PTrelax}
	We explain more about the correlation between these two tree structures. Note that, by definition, every pendant pair of a graph $G$ is $\delta(G)$-edge-connected. Let $T$ be a $\delta$-edge-connectivity tree and $T'$ be a pendant tree of $G$. Since every block of $T'$ is $\delta$-edge-connected, it must be contained in some block of $T$. Hence, the vertex partition of every pendant tree \emph{refines} the partition of $V$ into $\delta$-edge-connected components. It is also not hard to see that, given a $\delta$-edge-connectivity tree $T$, there is a pendant tree $T'$, such that contracting all edges $e\in E(T')$ with $c(e)\geq\delta$ gives $T$. Hence, most of the results derived by exploiting pendant trees can be reformulated to statements about $\delta$-edge-connected pairs. In particular, Lemma~\ref{lem:leaf} gives the following corollary.
	
	\begin{corollary}\label{cor:edgeblocks}
		Every simple graph $G$ contains a set $S$ of at least $\delta(G)+1$ vertices such that $\lambda_G(v,w) \geq \delta(G)$ for every $v,w \in S$.
	\end{corollary}
	
	More generally, Theorems~\ref{thm:independent} still holds without further ado when we replace the binary relation pendant pair by the $\delta$-edge-connectivity relation on vertex pairs, that yields Theorem~\ref{thm:fewcomponents}.

	\subsection{Tightness}
	Corollary~\ref{cor:edgeblocks} is tight by the graph $G$ that was constructed in Section~\ref{sec:tightness1}. The bounds of Theorems~\ref{thm:independent} and~\ref{thm:main} are asymptotically tight by considering the unions of $\frac{n}{\delta+1}$ many disjoint cliques $K_{\delta+1}$.
	
	Each of the conditions $\delta \geq 5$, $\lambda \geq 4$ and $\kappa \geq 3$ in Theorems~\ref{thm:independent} and~\ref{thm:main} is tight, as the graph in Figure~\ref{fig:delta4} can be arbitrarily large and satisfies $\delta = 4$, $\lambda = 3$ and $\kappa = 2$, but has only a constant number of pendant pairs. Also the simpleness condition in both results is indispensable: Consider the path graph on $n$ vertices in which the two end edges have multiplicity $\delta$ and all other edges have multiplicity $\delta/2$. This graph has precisely $2$ pendant pairs, each at one of its ends.
	
	\begin{figure}[h!t]
		\centering
		\includegraphics[scale=0.7]{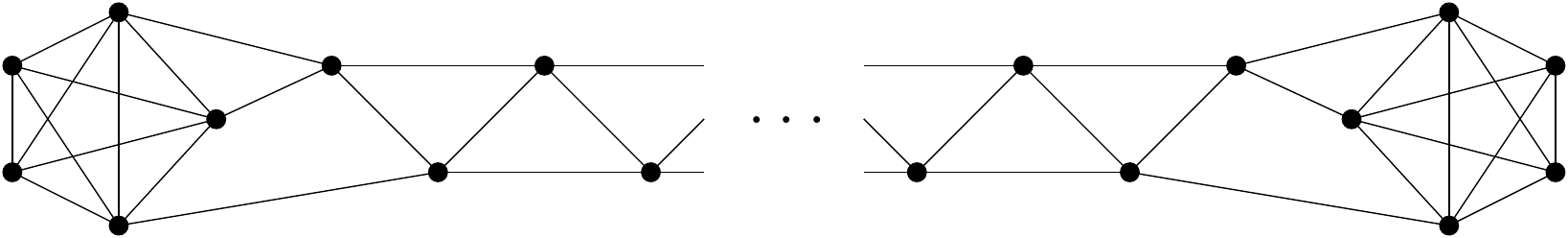}
		\caption{The \emph{bone graph} $G$, whose only pendant pairs are the ones contained in the two $K_5$ (those form the only leaf blocks of the pendant pair tree). Hence, $G$ has exactly $20$ pendant pairs.}
		\label{fig:delta4}
	\end{figure}
	
	\bibliographystyle{abbrv}
	\bibliography{lower_bound_of_pendant_pairs_ref}
	
	\newpage

	\appendix

	\section{Details and Proofs of Section~\ref{sec:improvement}}\label{sec:appcbs}
	Here we give the proofs omitted in Section~\ref{sec:improvement}. We first give the construction of the non-trivial min-cut tree.
	
	\begin{lemma}\label{lem:expand}
		Let $C \subset V(G)$ be a non-trivial min-cut of a simple graph $G$. Then $|C| \geq \delta(G)$.
	\end{lemma}
	\begin{proof}
		Let $p := |C|$. Then $\delta \geq \lambda \geq \sum_{v\in C}(d(v)-(p-1)) \geq p\delta - p(p-1)$ implies $p \geq \delta$, as $p>1$.
	\end{proof}
	
	For two vertices $s$ and $t$, let $\lambda_G^*(s,t) := \min\{d(A) \mid A \text{ is an }s\text{-}t\text{-cut of } G\}$ be the \emph{minimum size} of a non-trivial edge-cut separating $s$ and $t$ in a graph $G$ or $\infty$ if no non-trivial $s$-$t$-cut exists. Like the $k$-edge-connectivity relation, the relation $\lambda^*(s,t) \geq k$ is transitive for every $k$:
	
	\begin{lemma}\label{Alem:triangle}
		Every three vertices $s$, $t$ and $u$ of a graph $G$ satisfy the inequality $\lambda^*(s,t) \geq \min\{ \lambda^*(s,u), \lambda^*(u,t)\}$.
	\end{lemma}
	\begin{proof}
		Every non-trivial cut separating $s$ from $t$ is a non-trivial cut that either separates $s$ from $u$ or $u$ from $t$.
	\end{proof}
	
	Now we are ready to prove the existence of non-trivial min-cut trees (when $\lambda(G) \neq 0,2$).
	
	\begin{customproposition}{\ref{prop:computationn}}\label{Aprop:computationn}
		Let $G$ be a simple graph with $\lambda(G) \neq 0, 2$. Then a non-trivial min-cut tree $T$ of $G$ can be computed in time $O(n\theta_{flow})$.
	\end{customproposition}
	\begin{proof}
		We maintain a tree $T$ that satisfies Condition~(i) and whose vertices (which we again call blocks) form a partition of $V(G)$. At the beginning, $T$ has only one block, namely $V(G)$, and thus satisfies Condition~(i). However, this block may contain vertices that are separated by a non-trivial min-cut, which violates Condition~(ii).
		
		During the algorithm, we maintain for each block $B$ an ordered list $v_1,v_2,\dots,v_{r_B}$ of its vertices. We additionally maintain an index $1 \leq t_B \leq r_B$ such that all pairs in $\{v_1,v_2,\dots,v_{t_B}\}$ are not separated by a non-trivial min-cut of $G$; initially, we set $t_B := 1$ for every new block.
		
		We iteratively choose any block $B$ that satisfies $t_B < r_B$ and check whether $G$ contains a non-trivial min-cut separating $v_{t_B}$ and $v_{t_B+1}$. This can be done in time $O(\theta_{flow})$ using one flow computation as follows. If $\lambda(v_{t_B},v_{t_B+1}) > \lambda$, there is no such cut. Otherwise, $\lambda(v_{t_B},v_{t_B+1}) = \lambda$; then contracting the strongly connected components of the residual network of the flow computation with the approach of Ball and Provan~\cite{Ball1983} allows to check in time $O(m)$ whether there are only trivial min-cuts separating $v_{t_B}$ and $v_{t_B+1}$.
		
		If $v_{t_B}$ and $v_{t_B+1}$ are not separated by a non-trivial min-cut (i.e.\ $\lambda^*(v_{t_B},v_{t_B+1}) > \lambda$), we claim that $\lambda^*(v_i,v_{t_B+1}) > \lambda$ for every $1 \leq i < t_B$. Note that, by definition of $t_B$, $\lambda^*(v_i,v_{t_B}) > \lambda$. The claim holds as Lemma~\ref{Alem:triangle} implies that $\lambda^*(v_i,v_{t_B+1}) \geq \min\{\lambda^*(v_i,v_{t_B}), \lambda^*(v_{t_B},v_{t_B+1})\} > \lambda$. Hence, we set $t_B := t_B +1$.
		
		Otherwise, $v_{t_B}$ and $v_{t_B+1}$ are separated by a non-trivial min-cut $X \subset V$ with $v_{t_B} \in X$. In that case, we claim that there is a non-trivial min-cut $Z \ni v_{t_B}$ that separates $v_{t_B}$ and $v_{t_B+1}$ and does not cross any cut induced by an edge of $T$. Assume for now $Z$ exists and can be found efficiently. Then, as done for the construction of the pendant tree, we split $B$ into the two blocks $B_0 \ni v_{t_B}$ and $B_1 \ni v_{t_B+1}$ and introduce the new edge $B_0B_1$ to $T$ such that the former neighbors of $B$ are adjacent to either $B_0$ or $B_1$; to which one, is determined by the non-crossing cut $Z$. In total, this takes time $O(\theta_{flow})$ per split. Since Condition~(ii) is satisfied after at most $n-1$ splits (in fact, after at most $\max\{0,n-2\delta+1\}$ splits due to Lemma~\ref{lem:leafn}), this gives the desired statement of the proposition.
		
		It remains to show that we can efficiently find a non-trivial min-cut $Z \ni v_{t_B}$ that separates $v_{t_B}$ and $v_{t_B+1}$ and does not cross any cut induced by an edge of $T$. If $|X \cap B| \geq 2$ and $|\overline{X} \cap B| \geq 2$, the claim follows immediately from applying Lemma~\ref{lem:noncrossing}. Thus, we assume without loss of generality $X \cap B = \{v_{t_B}\}$. In particular, $X \cap \overline{B} \neq \emptyset$, as the cut $X$ is non-trivial, and hence $B \notin V_0$ (i.e.\ $B$ has degree at least one in $T$).
		
		We first consider the case $B\in V_1$. Assume that $\lambda = 1$. If $X$ and $B$ cross, $d(X\cap B) = d(\overline{X}\cap B) = d(X\cap\overline{B}) = d(\overline{X}\cap \overline{B}) = \lambda = 1$ by Lemma~\ref{lem:strongernc}. Thus, there is exactly one edge leaving $X \cap B$; let $A$ be the vertex set in $\{\overline{X}\cap B, X\cap\overline{B}, \overline{X}\cap \overline{B}\}$ where this edge ends. Then, $d((X \cap B) \cup A) = 0$, which violates $\lambda=1$. Hence, $X$ does not cross $B$ and we can simply take $Z := X$.
		
		If $\lambda \neq 1$, then, by the given condition $\lambda \neq 0, 2$, we have $\lambda \geq 3$. By Lemma~\ref{lem:expand}, $|B| \geq \delta \geq \lambda \geq 3$. Hence, $|\overline{X}\cap B| \geq 2$. Therefore, by Lemma~\ref{lem:strongernc}, $\overline{X}\cap B$ is a non-crossing minimum cut and this cut is non-trivial, and we therefore take $Z:= \overline{X}\cap B$.
		
		We now consider the case $B\in V_r$ such that $r\geq 2$. Let $A_1,\dots,A_r$ be the neighbors of $B$ in $T$ and write $C_i:= C_{A_i B}$ for every $i \in \{1,\dots,r\}$. Let without loss of generality $C_1,\dots,C_t$ ($0 \leq t\leq r$) be exactly those $C_i$ with $C_i\subset X$.
		
		If $t > 0$, i.e.\ there is some $i$ such that $C_i\subset X$. Then the cut $Z := X - \bigcup_{i>t} C_i$ is clearly non-trivial. Since $C_i\not\subset X$ for $i>t$, we can iteratively apply Lemma~\ref{lem:strongernc} to $C_i$ for every $i>t$, which trims $X$ to $Z$. Hence, $Z$ is the desired non-crossing non-trivial min-cut.
		
		If $t=0$, i.e.\ $C_i\not\subset X$ for every $i\in \{1,\dots,r\}$. Since $X$ is non-trivial and $|X\cap B|=1$, $X\cap \overline{B}$ is non-empty. Hence, there is some cut, say $C_1$, satisfying $X\cap C_1\neq\emptyset$. Set $Z:= (X\cup C_1) - \bigcup_{i\neq 1} C_i$; hence $Z$ is non-crossing. Since $r\geq 2$, $Z$ is non-trivial. Applying Lemma~\ref{lem:strongernc} to $\overline{X}$ and $\overline{C_1}$ shows that $\overline{X} \cap \overline{C_1} = X \cup C_1$ is a minimum cut. Now applying Lemma~\ref{lem:strongernc} iteratively to $C_i$ for every $i \neq 1$ trims $X \cup C_1$ to $Z$. Hence, $Z$ is the desired non-crossing non-trivial min-cut.
		
		Algorithmically, the case distinction can be computed in time $O(m)$ and every single case can be computed in time $O(\theta_{flow})$ by using a flow routine to obtain non-crossing cuts after appropriate contractions. This gives the overall running time $O(n\theta_{flow})$.
	\end{proof}
	
	Now we give the proofs of other technical lemmas.
	
	\begin{customlemma}{\ref{lem:leafn}}\label{Alem:leafn}
		Let $T$ be a non-trivial min-cut tree of a simple graph $G$. Then every leaf block $A$ of $T$ satisfies $|A| \geq \delta(G)$.
	\end{customlemma}
	\begin{proof}
		It follow immediately from Lemma~\ref{lem:expand}.
	\end{proof}
	
	\begin{customlemma}{\ref{lem:bignonsingletonsn}}\label{Alem:bignonsingletonsn}
		Let $T$ be a non-trivial min-cut tree of a simple graph $G$. Let $A'A,AB,BB'$ be edges in $T$ such that $A,B \in V_2$. If $|A|+|B|>2$, $|A| + |B| \geq \delta(G)/2$.
	\end{customlemma}
	\begin{proof}
		Let $p := |A|$ and $q := |B|$. It is clear that $\sum_{v\in A\cup B} d(v,C_{A'A}) \leq \lambda \leq \delta$, $\sum_{v\in A\cup B} d(v,C_{B'B})\leq \delta$ and $d(v,C_{A'A})+d(v,C_{B'B})\geq d(v)-(p+q-1)$. Therefore, $2\delta \geq \sum_{v\in A\cup B} (d(v,C_{A'A})+d(v,C_{B'B})) \geq \sum_{v\in A\cup B} (d(v)-(p+q-1)) \geq (p+q)(\delta -(p+q-1))$, which gives $p+q \geq \frac{p+q-2}{p+q-1}\cdot \delta \geq \frac{1}{2}\cdot \delta$ if we assume $p+q > 2$.
	\end{proof}
	
	\begin{customlemma}{\ref{lem:rstarn}}\label{Alem:rstarn}
		Let $T$ be a non-trivial min-cut tree of a simple graph $G$. Let $A$ be a block in $V_r$ with neighborhood $B_1,\dots,B_r\in V_2$ in $T$ such that $|A|=|B_1|= \dots = |B_r| = 1$. Then $\delta(G) \leq r^2 + r$.
	\end{customlemma}
	\begin{proof}
		Let $A:=\{v_A\}$. For every $1 \leq i \leq r$, let $B_i:=\{v_i\}$ and $B_i'\neq A$ be the block that is adjacent to $B_i$ in $T$. Write $C_i := C_{B'_iB_i}$. Since every $v_i$ or $v_A$ can have at most $r$ neighbors in $\{v_A,v_1,\ldots,v_r\}$, we have, for every $1 \leq i \leq r$, $d(v_i) \leq r + \sum_{j=1}^r d(v_i,C_j)$, and $d(v_A) \leq r + \sum_{i=1}^r d(v_A,C_i)$. On the other hand, we have, for every $1 \leq i \leq r$,  $d(v_A,C_i)+\sum_{j=1}^r d(v_j,C_i) \leq \lambda \leq \delta$. Therefore, $\sum_{i=1}^r (\delta + d(v_A,C_i) + \sum_{j=1}^r d(v_j,C_i) ) \leq \sum_{i=1}^r ( d(v_i) + d(v_A,C_i) + \sum_{j=1}^r d(v_j,C_i) ) \leq \sum_{i=1}^r ( r + \sum_{j=1}^r d(v_i,C_j) + \delta )$, which implies $r^2\geq \sum_{i=1}^r d(v_A,C_i) \geq d(v_A) - r \geq \delta - r$.
	\end{proof}

	\section{Details and Proofs of Section~\ref{sec:LBPP}}\label{sec:LBPPA}
	Here we give the construction of the pendant tree, as well as details of the necessary technical lemmas that are used to prove the lower bound on pendant pairs.
	
	\begin{customproposition}{\ref{prop:computation}}\label{Aprop:computation}
		A pendant tree $T$ of a graph $G$ can be computed in time $O(n \theta_{cut})$.
	\end{customproposition}
	\begin{proof}
		We first show how to find a tree that satisfies Conditions~(i) and~(iii). We maintain a partition of $V$ that is the vertex set of a tree $T$ and iteratively split one of its blocks into two blocks. At the beginning, the partition is $\{V\}$ and thus $T$ consists of the single block $V$. Clearly, $T$ satisfies Condition~(iii); however, the block may contain non-pendant pairs, which violates Condition~(i). During the algorithm, we maintain for each block $B$ a sorted list of its vertices $v_1,v_2,\dots,v_{r_B}$ such that $d(v_1) \geq d(v_2) \geq \dots \geq d(v_{r_B})$. We additionally maintain an index $1 \leq t_B \leq r_B$ such that all pairs in $\{v_1,v_2,\dots,v_{t_B}\}$ are pendant; initially, we set $t_B := 1$ for every new block.
		
		We iteratively choose any block $B$ that satisfies $t_B < r_B$ and check whether $\{v_{t_B},v_{t_B+1}\}$ is pendant in time $O(\theta_{flow})$ by comparing $\lambda(v_{t_B},v_{t_B+1})$ with the respective vertex degrees. If so, we set $t_B := t_B +1$, as every $\{v_i,v_{t_B+1}\}$ for $1 \leq i < t_B$ is pendant as well, since $\lambda(v_i,v_{t_B+1}) \geq \min\{\lambda(v_i,v_{t_B}),\lambda(v_{t_B},v_{t_B+1})\} = \min\{d(v_{t_B}),d(v_{t_B+1})\} = \min\{d(v_i),d(v_{t_B+1})\}\geq \lambda(v_i,v_{t_B+1})$ by transitivity of $\lambda$ and the definition of pendant pairs.
		
		Otherwise, we apply Lemma~\ref{lem:noncrossing} to find a minimum $v_{t_B}$-$v_{t_B+1}$-cut $X$, say with $v_{t_B} \in X$, that does not cross any cut induced by an edge of $T$. Algorithmically, this cut can be computed by contracting the vertices in the blocks of every component of $T-B$, as done for the construction of Gomory-Hu trees. The cut splits the block $B$ into two blocks $B_0\ni v_{t_B}$ and $B_1\ni v_{t_B+1}$. We introduce the new edge $B_0B_1$ in $T$ such that the former neighbors of $B$ are adjacent to either $B_0$ or $B_1$; to which, is determined by the non-crossing cut $X$. The vertices $v_1,\ldots,v_{t_B}$ will be in $B_0$, since $d(X,\overline{X}) < d(v_{t_B+1}) \leq \min\{d(v_i),d(v_{t_B})\} = \lambda(v_i,v_{t_B})$ for every $1 \leq i < t_B$.
		
		Note that the new vertex lists for $B_0$ and $B_1$ inherit their order from the list of $B$; hence, it suffices to sort the vertices of the single block $V$ at the very beginning in time $O(n+m)$ using bucket sort. Clearly, this algorithm terminates in time $O(n \theta_{flow})$ with a tree that satisfies Condition~(i).
		
		We claim that Condition~(iii) holds throughout the algorithm. The new tree-edge $B_0B_1$ satisfies this condition with representative vertices $v_{t_B}$ and $v_{t_B+1}$. We can reuse the old representatives for every other tree-edge, except for the tree-edges $B_jA$ with $j \in \{0,1\}$ and $A \notin \{B_0,B_1\}$ that had representatives $a^\ast \in A$ and $b^\ast \in B_{1-j}$ before $B$ was split. For these tree edges, we show that $v_{t_B+j}$ and $a^\ast$ are new representatives for $B_jA$ satisfying Condition~(iii). Since $BA$ satisfied Condition~(iii) before, it suffices to prove $\lambda_G(v_{t_B+j},a^\ast) = \lambda_G(b^\ast,a^\ast)$. Note that $\lambda_G(v_{t_B+j},a^\ast) \leq \lambda_G(b^\ast,a^\ast)$, as $v_{t_B+j}$ and $a^\ast$ are separated by the cut induced by $BA$. Moreover, by Lemma~\ref{lem:noncrossing} and transitivity of $\lambda$, we have
		\begin{align*}
			\lambda_G(v_{t_B+j},a^\ast)&=\lambda_{G/C_{B_{1-j}B_j}}(v_{t_B+j},a^\ast) \tag{Lemma~\ref{lem:noncrossing}; $C_{B_{1-j}B_j}$ is contracted into vertex $\{C_{B_{1-j}B_j}\}$}\\
			&\geq \min\{\lambda_{G/C_{B_{1-j}B_j}}(v_{t_B+j}, \{C_{B_{1-j}B_j}\}), \lambda_{G/C_{B_{1-j}B_j}}(\{C_{B_{1-j}B_j}\},a^\ast)\} \tag{by transitivity of $\lambda$}\\
			&=\min\{\lambda_G(v_{t_B},v_{t_B+1}), \lambda_{G/C_{B_{1-j}B_j}}(\{C_{B_{1-j}B_j}\},a^\ast)\} \tag{by choice of $\{C_{B_{1-j}B_j}\}$}\\
			&\geq \lambda_G(b^\ast,a^\ast). \tag{as every $\{C_{B_{1-j}B_j}\}$-$a^\ast$-cut in $G/C_{B_{1-j}B_j}$ is a $b^\ast$-$a^\ast$-cut in $G$}
		\end{align*}
		
	Thus, a tree $T'$ satisfying Conditions~(i) and~(iii) can be computed in time $O(n \theta_{flow})$. In order to compute a pendant tree from $T'$, we check for every edge $AB$ in $T'$ whether there is a non-pendant pair $\{a,b\}$ with $a \in A$ and $b \in B$. According to Lemma~\ref{lem:representative}, it suffices to test whether two arbitrary vertices $a_{max}$ and $b_{max}$ of maximum degree in $A$ and $B$ are pendant (this needs time $O(\theta_{flow})$ per edge). If so, we contract $AB$; by transitivity of $\lambda$, the resulting graph still satisfies Conditions~(i) and~(iii). Hence, a pendant tree can be computed in time $O(n\theta_{flow})$.
	\end{proof}
	
	We note that the same approach can be used to construct a $\delta$-edge-connectivity tree in the same running time $O(n \theta_{flow})$. Next, we give the proofs of the technical lemmas that show that blocks of pendant trees have average size $\Omega(\delta)$, see Section~\ref{sec:largeblockPT} for more discussion.
		
	\begin{customlemma}{\ref{lem:leaf}}\label{Alem:leaf}
		Every leaf block $A$ of a pendant tree $T$ of a simple graph $G$ satisfies $|A| > \delta(G)$.
	\end{customlemma}
	\begin{proof}
		Let $p:=|A| \geq 1$ and let $B$ be the block adjacent to $A$ in $T$. We have $\max_{v\in A} d(v) > c(AB) \geq \sum_{v\in A}(d(v)-(p-1)) \geq \max_{v\in A} d(v) + \delta(p-1) - p(p-1)$, where the last inequality singles out the maximum degree. Therefore, $p>1$ and $p>\delta$.
	\end{proof}
	
	\begin{lemma}\label{Alem:bignonsingletons}
		Let $T$ be a pendant tree of a simple graph $G$. Let $A'A,AB,BB'$ be edges in $T$ such that $A,B \in V_2$. If $|A|+|B|>2$, $|A|+|B| \geq \delta(G)-1$.
	\end{lemma}
	\begin{proof}
		Let $p := |A|$ and $q := |B|$. By Lemma~\ref{lem:maxdeg}, we have $\sum_{v\in A\cup B} d(v,C_{A'A}) \leq c(A'A) \leq \max_{v\in A} d(v)-1$ and $\sum_{v\in A\cup B} d(v,C_{B'B})\leq \max_{v\in B} d(v)-1$. For $v\in A\cup B$, there are at most $p+q-1$ edges that are incident to $v$ and $A \cup B$, which implies $d(v,C_{A'A})+d(v,C_{B'B})\geq d(v)-(p+q-1)$. Therefore, $\max_{v\in A} d(v)+\max_{v\in B} d(v)-2\geq \sum_{v\in A\cup B} (d(v,C_{A'A})+d(v,C_{B'B})) \geq \sum_{v\in A\cup B} (d(v)-(p+q-1)) \geq \max_{v\in A} d(v)+\max_{v\in B} d(v) + (p+q-2)\delta-(p+q)(p+q-1)$, which gives $(p+q)(p+q-1) \geq (p+q-2)\delta+2$ and thus $(p+q)(p+q-2) \geq (p+q-2)(\delta-1)$. Hence, $p+q \geq \delta-1$ if $p+q > 2$.
	\end{proof}
	
	\begin{lemma}\label{Alem:rstar}
		Let $T$ be a pendant tree of a simple graph $G$ with $|V(T)|>1$. Let $A=\{v_A\}$ be a block in $V_r$ with neighborhood $B_1,\dots,B_r\in V_2$ in $T$ such that $|A|=|B_1|= \dots = |B_r| = 1$. Let $B_i'\neq A$ be the block that is adjacent to $B_i$ in $T$. Then $d(v_A) \leq r^2-2\gamma$, where $\gamma := \sum_{1 \leq i < j \leq r} d(C_{B'_iB_i},C_{B'_jB_j})$ is the number of \emph{cross-edges}. In particular, we have $\delta(G) \leq r^2$ and $\lambda(G) < r^2$. Moreover, if $r = 2$, $\kappa(G) \leq 2$.
	\end{lemma}
	\begin{proof}
		Note that $r\geq2$, since there is no singleton leaf block (Corollary~\ref{cor:nosingletonleaf}). For every $1 \leq i \leq r$, let $B_i=\{v_i\}$ and $C_i := C_{B'_iB_i}$. Since every $v_i$ or $v_A$ can have at most $r$ neighbors in $\{v_A,v_1,\ldots,v_r\}$, we have $d(v_i) \leq r + \sum_{j=1}^r d(v_i,C_j)$ for every $1 \leq i \leq r$, and $d(v_A) \leq r + \sum_{i=1}^r d(v_A,C_i)$. On the other hand, by Lemma~\ref{lem:maxdeg}, we have, for every $1 \leq i \leq r$, $d(v_A,C_i)+\sum_{j=1}^r d(v_j,C_i) + \sum_{j \in \{1,\dots,r\}-i} d(C_j,C_i) \leq d(C_i) \leq d(v_i)-1$ (see Figure~\ref{fig:rsquare}). Therefore,
		\begin{align*}
		&\sum_{i=1}^r \left(d(v_i) + d(v_A,C_i) + \sum_{j=1}^r d(v_j,C_i) + \sum_{j \neq i} d(C_j,C_i)\right)\\
		& \hspace{6cm} \leq \sum_{i=1}^r \left(r + \sum_{j=1}^r d(v_i,C_j) + d(v_i)-1\right)\\
		\Leftrightarrow & \sum_{i=1}^r \left( d(v_A,C_i) + \sum_{j \neq i} d(C_j,C_i) \right) \leq r^2-r\\
		\Leftrightarrow & \sum_{i=1}^r d(v_A,C_i) \leq r^2-r-2\gamma,
		\end{align*}and hence, $d(v_A) \leq r + \sum_{i=1}^r d(v_A,C_i) \leq r^2-2\gamma$. In particular, this gives $\delta \leq r^2$ and, according to Lemma~\ref{lem:maxdeg}, $\lambda \leq c(A B_1) < d(v_A) \leq r^2$.
		
		Now, we claim that, if $r = 2$, then $\kappa \leq 2$. If $\gamma > 0$, then $\kappa \leq \delta \leq d(v_A) \leq r^2-2\gamma \leq 2$. If $\gamma = 0$, let $S := \{v_A,v_1,v_2\}$, which is a separator of $G$ of size $3$. If a vertex $z \in S$ has no neighbor in $C_i$ for some $1 \leq i \leq 2$, $S - z$ is a separator of size $2$, which gives the claim. Otherwise, we have $c(B_1A) \geq 3$ and $c(AB_2) \geq 3$, and in addition, $c(B_1A) = c(AB_2) = 3$, according to $d(v_A) \leq 4$ and Lemma~\ref{lem:maxdeg}. Hence, $v_A$ is of degree 2 in $G$, which gives the claim.
	\end{proof}
	
	Setting $r=2$ in Lemma~\ref{Alem:rstar} gives the following corollary for adjacent blocks of 2-paths. Note that the proof of Lemma~\ref{Alem:rstar} allows to weaken the conditions of this corollary further if the number of cross-edges is large.
	
	\begin{corollary}\label{Acor:innerblocks}
		Let $G$ be simple and let $AB$ and $BC$ be edges in a 2-path of $T$. If $\delta(G) \geq 5$ or $\lambda(G) \geq 4$ or $\kappa(G) \geq 3$, then $|A|+|B|+|C|>3$.
	\end{corollary}

	Now we are ready to show that the blocks of 2-paths contain many vertices if $\delta(G) \geq 5$ or $\lambda(G) \geq 4$ or $\kappa(G) \geq 3$.
	
	\begin{customcorollary}{\ref{lem:big2paths}}\label{Alem:big2paths}
		Let $T$ be a pendant tree of a simple graph $G$ satisfying $\delta(G) \geq 5$ or $\lambda(G) \geq 4$ or $\kappa(G) \geq 3$. Let $P$ be a 2-path of $T$. Then $\sum_{S\in V(P)} |S| \geq (|V(P)|-2)\frac{\max\{4,\delta(G)\}}{3}+2$.
	\end{customcorollary}
	\begin{proof}
		For any two consecutive edges $AB$ and $BC$ in $P$, applying Corollary~\ref{Acor:innerblocks} gives $|A|+|B|+|C|>3$. Due to Lemma~\ref{Alem:bignonsingletons}, this implies $|A|+|B|+|C|\geq \max\{4,\delta\}$. Since there may be at most two singletons that are not contained in such a triple, we conclude that $$\sum_{S\in V(P)}|S| \geq (|V(P)|-2)\frac{\max\{4,\delta\}}{3}+2.$$
	\end{proof}
	
	We will use these results on large blocks to obtain our main Theorems~\ref{thm:independent} and~\ref{thm:main}. While the latter shows the existence of $\Omega(\delta n)$ pendant pairs, as mentioned in the introduction, the former gives an upper of the number of blocks of pendant tree.
	
	\begin{customtheorem}{\ref{thm:independent}}\label{Athm:independent}
		Let $G$ be a simple graph that satisfies $\delta(G) \geq 5$ or $\lambda(G) \geq 4$ or $\kappa(G) \geq 3$. Let $T$ be a pendant tree of $G$. Then $T$ has at most $\frac{12}{\delta + 12} n = O(n/\delta)$ blocks.
	\end{customtheorem}
	\begin{proof}
		Let $T$ be a pendant tree of $G$. We can assume that $|V(T)| \geq 2$. Using the previous results, we can relate $n$ with the number of blocks in $T$ by considering the blocks in $V_2$ separately as follows:
		\begin{align*}
		n &=\sum_{S\in V(T)} |S|\\
		&=|V(T)|+\sum_{S\in V_1\cup V_{>2}} \left(|S|-1\right) +\sum_{\textrm{2-path }P}\left(\sum_{S\in V(P)}\left(|S|\right)-|V(P)|\right)\\
		&\geq |V(T)|+|V_1|\delta+\sum_{\textrm{2-path }P,\ |V(P)| \geq 3} \left(|V(P)|-2\right)\left(\frac{\max\{4,\delta\}}{3}-1\right) \tag{by Lemmas~\ref{lem:leaf} and~\ref{lem:big2paths}}\\
		&\geq |V(T)|+|V_1|\delta+\frac{1}{12}|V_2^{in}|\delta \tag{as $\delta=3 \Rightarrow (\frac{4}{3}-1) > \frac\delta{12}$ and $\delta \geq 4 \Leftrightarrow \frac{\delta}{3}-1\geq\frac\delta{12}$}\\
		&\geq |V(T)|+\frac{1}6(|V_1|+|V_2^{out}|+|V_{>2}|)\delta+\frac{1}{12}|V_2^{in}|\delta \\ \tag{since $6|V_1| \geq |V_1|+4|V_2^{out}|+|V_{>2}|$ by Lemma~\ref{lem:leafbound}}\\
		&\geq\left(1+\frac{1}{12}\delta\right)|V(T)|.
		\end{align*} Hence, $|V(T)| \leq (1+\frac{1}{12}\delta)^{-1}n=O(n/\delta)$.
	\end{proof}
	
	This already shows that $G$ contains at least $n-|V(T)| \geq (1-\frac{12}{\delta+12})n = \Omega(n)$ pendant pairs. We now improve this lower bound to $\Omega(\delta n)$ in the following theorem. This is done by grouping the blocks in a more subtle way. The main idea is that the blocks are of average size $\Omega(\delta)$ and therefore contain $\Omega(\delta^2)$ pendant pairs on average. As the number of blocks is $O(n/\delta)$, we thus expect that the number of pendant pairs is $\Omega(\frac{n}{\delta} \cdot \delta^2) = \Omega(\delta n)$.
	
	\begin{customtheorem}{\ref{thm:main}}\label{Athm:main}
		Let $G$ be a simple graph that satisfies $\delta(G) \geq 5$ or $\lambda(G) \geq 4$ or $\kappa(G) \geq 3$. Then $G$ contains at least $\frac{1}{30}\delta n = \Omega(\delta n)$ pendant pairs.
	\end{customtheorem}
	\begin{proof}
		Note that $n > \delta \geq 3$. If $G$ does not contain a non-pendant pair, there are $\binom{n}{2} \geq \frac{\delta n}{30}$ pendant pairs in $G$. Otherwise, $G$ contains a non-pendant pair. Let $T$ be a pendant tree of $G$; then $|V(T)| \geq 2$.
		
		For each 2-path $P$ with $|V(P)| \geq 3$, let $P^*$ be a subpath obtained from $P$ by deleting at most two endblocks (that is, blocks in $P \cap V_2^{out}$) of $P$ such that $|V(P^*)|$ is a multiple of $3$. Then, we split $P^*$ into subpaths $P^*_{1},\dots,P^*_{|V(P^*)|/3}$, each of size 3. Now, let $M_P$ be a collection of blocks that contains exactly one block $S_i \in V(P^*_{i})$ for every $i = 1,\dots,\frac{|V(P^*)|} 3$, such that $S_i$ is of maximum size amongst other blocks in $V(P^*_{i})$. By Corollary~\ref{Acor:innerblocks} and Lemma~\ref{Alem:bignonsingletons}, every block $S \in M_P$ is of size at least $\max\{2,(\delta-1)/2\}$.
		
		Let $V_2^* := V_2 - \bigcup_{\textrm{2-path }P, |V(P)| \geq 3}V(P^*)\subseteq V_2^{out}$. For every leaf block $S \in V_1$, let $Y_S$ be a collection of blocks that consists of $S$, at most four blocks from $V_2^*$ and at most one block from $V_{>2}$ such that $Y_S$ ($S \in V_1$) form a partition of $V_1\cup V_2^*\cup V_{>2}$; such allocation exists as $|V_2^*|\leq|V_2^{out}|\leq 4|V_1|$ and $|V_{>2}|\leq |V_1|$ (Lemma~\ref{lem:leafbound}). For every $S\in V_1$, let $D_S$ be a block in $Y_S$ of maximum size. Then, by Lemma~\ref{lem:leaf}, $|D_S| \geq |S| > \delta$. 
		
		Now we can count the number of pendant pairs to obtain the desired lower bound, as the blocks have average size $\Omega(\delta)$. The number of pendant pairs in $G$ is at least
		\begin{align*}
		&\sum_{S\in V(T)}\binom{|S|}2\\
		&\geq \sum_{S\in V_1}\frac{|D_S|(|D_S|-1)}2 + \sum_{\textrm{2-path }P,\ |V(P)| \geq 3\ }\sum_{S\in M_P}\frac{|S|(|S|-1)}2\\
		&\geq \frac{\delta}{2}\sum_{S\in V_1}|D_S| + \frac{\delta}{10}\sum_{\textrm{2-path }P,\ |V(P)| \geq 3\ }\sum_{S\in M_P}|S| \tag{as $|D_S| > \delta$, $|S|\geq\max\{2,\frac{\delta-1}2\}$ and $\delta\geq 3$}\\
		&\geq \frac{\delta}{2}\cdot\frac{1}{6} \sum_{S\in V_1\cup V_2^*\cup V_{>2}}|S| + \frac{\delta}{10}\cdot\frac{1}{3} \sum_{S \in V_2 - V_2^*}|S|\\
		&\geq \frac{1}{30}\delta n = \Omega(\delta n).
		\end{align*}
	\end{proof}

\end{document}